\documentclass{amsart}
\usepackage{amsrefs}

\usepackage{amsfonts}
\usepackage{amsmath}
\usepackage{amssymb}
\usepackage{amsthm}
\usepackage{latexsym}
\usepackage{url}
\usepackage{verbatim}
\usepackage{mathtools}

\theoremstyle{plain}
\newtheorem{theorem}{Theorem}
\newtheorem{proposition}[theorem]{Proposition}

\newtheorem{corollary}[theorem]{Corollary}

\theoremstyle{definition}
\newtheorem{definition}{\mdseries\scshape Definition}

\newtheorem{remark}{\mdseries\scshape Remark}

\BibSpec{article}{%
    +{}  {\PrintAuthors}                {author}
    +{,} { \bibquotes}                     {title}
    +{.} { }                            {part}
    +{:} { \bibquotes}                     {subtitle}
    +{,} { \PrintContributions}         {contribution}
    +{.} { \PrintPartials}              {partial}
    +{,} { }                            {journal}
    +{}  { \textbf}                     {volume}
    +{} { \parenthesize}                  {number}
    +{,} { \eprintpages}                {pages}
    +{\mbox{}}  { \PrintDatePV}                {date}
    +{.} { }                            {status}
    +{.} { \eprint}        {eprint}
    +{}  { \parenthesize}               {language}
    +{}  { \PrintTranslation}           {translation}
    +{;} { \PrintReprint}               {reprint}
    +{.} { }                            {note}
    +{.} {}                             {transition}
    +{}  {\SentenceSpace \PrintReviews} {review}
}

\BibSpec{collection.article}{%
    +{}  {\PrintAuthors}                {author}
    +{,} { \bibquotes}                     {title}
    +{.} { }                            {part}
    +{:} { \bibquotes}                     {subtitle}
    +{,} { \PrintContributions}         {contribution}
    +{,} { in \textit}                            {booktitle}
   +{,} { pp.~}                        {pages}
    +{,} { in \PrintConference}            {conference}
     +{}  {\PrintBook}                   {book}   
    +{,} { \PrintDateB}                 {date}
    +{,} { }                            {status}
    +{,} { \PrintDOI}                   {doi}
    +{,} { \eprint}        {eprint}
    +{}  { \parenthesize}               {language}
    +{}  { \PrintTranslation}           {translation}
    +{;} { \PrintReprint}               {reprint}
    +{).} { }                            {note}
    +{).} {}                             {transition}
    +{}  {\SentenceSpace \PrintReviews} {review}
}

\BibSpec{book}{%
    +{}  {\PrintPrimary}                {transition}
    +{,} { \textit}                     {title}
    +{.} { }                            {part}
    +{:} { \textit}                     {subtitle}
    +{,} { \PrintEdition}               {edition}
    +{}  { \PrintEditorsB}              {editor}
    +{,} { \PrintTranslatorsC}          {translator}
    +{,} { \PrintContributions}         {contribution}
    +{,} { }                            {series}
    +{,} { \voltext}                    {volume}
    +{} { (}                            {publisher}
    +{,} { }                            {organization}
    +{,} { }                            {address}
    +{,} { \PrintDateB}                 {date}
    +{,} { }                            {status}
    +{}  { \parenthesize}               {language}
    +{}  { \PrintTranslation}           {translation}
    +{;} { \PrintReprint}               {reprint}
    +{).} { }                            {note}
    +{).} {}                             {transition}
    +{}  {\SentenceSpace \PrintReviews} {review}
}

\BibSpec{innerbook}{%
    +{} {\textit}                            {title}
    +{.} { }                            {part}
    +{:} { \textit}                            {subtitle}
    +{,} { \PrintEdition}               {edition}
    +{,} { \PrintTranslatorsC}          {translator}
    +{,} { \PrintContributions}         {contribution}
    +{,} { }                            {series}
    +{,} { \voltext}                    {volume}
    +{} { (}                            {publisher}
    +{,} { }                            {organization}
    +{,} { }                            {address}
    +{,} { \PrintDateB}                 {date}
    +{).} { }                            {note}
}

\BibSpec{thesis}{%
    +{}  {\PrintAuthors}                {author}
    +{,} { \textit}                     {title}
    +{:} { \textit}                     {subtitle}
    +{,} { \PrintThesisType}            {type}
    +{,} { (}                            {organization}
    +{,} { }                            {address}
    +{,} { \PrintDateB}                 {date}
    +{,} { \eprint}                     {eprint}
    +{,} { }                            {status}
    +{}  { \parenthesize}               {language}
    +{}  { \PrintTranslation}           {translation}
    +{;} { \PrintReprint}               {reprint}
    +{).} { }                            {note}
    +{).} {}                             {transition}
    +{}  {\SentenceSpace \PrintReviews} {review}
}

\BibSpec{report}{%
    +{}  {\PrintPrimary}                {transition}
    +{,} { \bibquotes}                     {title}
    +{.} { }                            {part}
    +{:} { \bibquotes}                     {subtitle}
    +{,} { \PrintEdition}               {edition}
    +{,} { \PrintContributions}         {contribution}
    +{,} { Technical Report }           {number}
    +{,} { }                            {series}
    +{} { (}                            {organization}
    +{,} { }                            {address}
    +{,} { \PrintDateB}                 {date}
    +{).} { \eprint}                     {eprint}
    +{,} { }                            {status}
    +{}  { \parenthesize}               {language}
    +{}  { \PrintTranslation}           {translation}
    +{;} { \PrintReprint}               {reprint}
    +{.} { }                            {note}
    +{.} {}                             {transition}
    +{}  {\SentenceSpace \PrintReviews} {review}
}

\newcommand{\set}[1]{\left\{#1\right\}}
\newcommand{\abs}[1]{\left\vert#1\right\vert}
\newcommand{\floor}[1]{\left\lfloor#1\right\rfloor}

\newcommand{\Z}{\mathbb{Z}}
\newcommand{\Zp}{\mathbb{Z}_p}
\newcommand{\oneunit}[1]{\left\langle#1\right\rangle}

\DeclareMathOperator{\ord}{ord}

\begin{document}

\title[Fixed points and rooted closed walks of
$x \mapsto x^{x^n}$]{Counting fixed points and rooted closed walks of the singular map
  $x \mapsto x^{x^n}$ modulo powers of a prime}

\author{Joshua Holden}

\address{Department of Mathematics,
Rose-Hulman Institute of Technology,
5500 Wabash Ave.,
Terre Haute, IN 47803, USA}

\email{holden@rose-hulman.edu}

\author{Pamela A. Richardson}

\address{Department of Mathematics and Computer Science,
Westminister College,
319 South Market Street, New Wilmington, 
PA 16172, USA}

\email{richarpa@westminister.edu}

\author{Margaret M. Robinson} 

\address{Department of Mathematics and Statistics, 
Mount Holyoke College,
50 College Street, South Hadley,
  MA 01075, USA}

\email{robinson@mtholyoke.edu}




\begin{abstract}
 The ``self-power'' map $x \mapsto x^x$ modulo $m$ and its generalized
 form $x \mapsto x^{x^n}$ modulo $m$ are of considerable interest for
 both theoretical reasons and for potential applications to
 cryptography.  In this paper, we use $p$-adic methods, primarily
 $p$-adic interpolation, Hensel's lemma, and lifting singular points
 modulo $p$, to count fixed points and rooted closed walks of equations related
 to these maps when $m$ is a prime power.  In particular, we introduce a new technique for lifting singular solutions of several congruences in several unknowns using the left kernel of the Jacobian matrix.
\end{abstract}

\maketitle

\section{Introduction}

The study of the ``self-power'' map $x \mapsto x^x$ modulo $m$ goes
back at least to two papers by Crocker in the 1960's~\cites{crocker66,
  crocker69}.  Its study has accelerated in recent years due to both
improvements in technique (see, for instance, \citelist{\cite{somer}%
  \cite{holden02}\cite{holden02a}\cite{holden_moree}\cite{REU2010}%
  \cite{holden_robinson}\cite{anghel}\cite{anghel16}\cite{balog_et_al}%
  \cite{CG14}\cite{CG15}\cite{G18}%
  \cite{holden_friedrichsen}\cite{kurlberg_et_al}\cite{felix_kurlberg}\cite{HHJ}})
and its relation to a variation of the ElGamal digital signature
scheme given in, e.g., \cite{handbook}*{Note~11.71}.  Most of these focused on the case where $m$ is a prime, but~\cite{HHJ} investigated solutions to
\begin{equation} \label{eqn:spfpm}
x^x \equiv x \pmod{m}
\end{equation}
for general composite $m$, and~\cite{holden_robinson} used $p$-adic techniques to
investigate solutions to the equations (among others)
\begin{equation}\label{spe}
  x^{x} \equiv c \pmod{p^e}
\end{equation} for fixed $c$ and $x$ in $\set{1, \dots, p^e(p-1)}$
and
\begin{equation} \label{hae}
    h^{h} \equiv a^{a} \pmod{p^e}
\end{equation}
for $a$ and $h$ in $\set{1, \dots, p^e(p-1)}$.  

In this work we will
use similar techniques to investigate the number of fixed points of the
self-power map, i.e., solutions to
\begin{equation} \label{eqn:spfp}
x^x \equiv x \pmod{p^e},
\end{equation}
and two-cycles, or solutions to
\begin{equation} \label{eqn:sptc}
x^x \equiv y \pmod{p^e} \quad \text{and} \quad y^y \equiv x \pmod{p^e},
\end{equation}
as well as solutions in the $p$-adic integers $\Zp$.
In fact, we give results for more general situations including
\begin{equation} \label{eqn:spfpn}
x^{x^n} \equiv x \pmod{p^e},
\end{equation}
\begin{equation} \label{eqn:sptcn}
x^{x^n} \equiv y \pmod{p^e} \quad \text{and} \quad y^{y^n} \equiv x \pmod{p^e},
\end{equation}
for all $p$ and $n$, 
and select cases of
\begin{equation} \label{eqn:sprwn}
x_1^{g(x_1)} \equiv x_2 \pmod{p^e}, \quad \ldots, \quad x_k^{g(x_k)} \equiv x_1 \pmod{p^e}.
\end{equation}

This particular generalization was inspired by study of the map $x
\mapsto g^{x^n}$ modulo $p$ for a fixed integer $g$, which has been
used in a secret sharing scheme~\cite{stadler} and a group signature
scheme~\cite{CS97}, among other places.  A preliminary study of the
case $n=2$ of this map was begun in~\cite{wood}, and the solutions to
$g^{x^n} \equiv x^k$ modulo $p^e$ were later studied in~\cite{mann}
with some conditions on $p$, $k$, and $n$.  It is also known that the
discrete logarithm problem, that is, the problem of inverting the map
$x \mapsto g^{x}$ modulo $p$, can be solved more quickly if a value of
$g^{x^n}$ modulo $p$ is known in addition.  (See \cite{cheon}, for
example.)  It would be interesting to know if this also applies to the
self-power map.  For a general polynomial $g(x)$, we also give some
results on the generalized self-power map $x \mapsto x^{g(x)}$ in the
case $e=1$.  Other results for this map, including discussions of
fixed points, appear in~\cite{kurlberg_et_al}*{Thm.~10},
\cite{CG14}*{Cor.~2}, and~\cite{G18}*{Cor.~1}.

Solutions to these congruences modulo $p^e$ can also be counted
without using $p$-adic techniques.  One advantage of using $p$-adic
methods is that we not only count solutions but also show how the
solutions modulo different values of $p^e$ relate.  In particular, we
show that almost all solutions fail to lift to arbitrarily high values
of $p^e$ (or equivalently to $\Zp$).  This is in stark contrast to the
situations in~\cite{holden_robinson}
and~\cite{mann}, where all
solutions lift arbitrarily high.

The primary $p$-adic techniques used in this paper are $p$-adic
interpolation and lifting techniques, including Hensel's lemma and
lifting singular points modulo $p$.  Unlike the situation
in~\cite{holden_robinson} and~\cite{mann}, not every solution modulo
$p$ is nonsingular.  Nonsingular solutions can be lifted uniquely to
$\Zp$ using Hensel's lemma, but singular ones cannot be lifted by the lemma.
Section~\ref{interpolation} provides the necessary background for the
$p$-adic techniques.  Section~\ref{fixed} counts the number of fixed
points, that is, solutions of~\eqref{eqn:spfpn}, for both odd $p$ and
$p=2$. Section~\ref{sec:rw} introduces rooted closed walks and some  techniques required for lifting solutions to systems of equations such as~\eqref{eqn:sprwn}.
Section~\ref{sec:tc} uses a new form of these techniques to count the number of
two-cycles, or solutions of~\eqref{eqn:sptcn}, for odd and even $p$.
Finally, Section~\ref{future} discusses future work.

\section{Interpolation and Lifting} \label{interpolation}

Let $p$ be a prime, and let $q=4$ if $p=2$, $q=p$ otherwise.  As
in~\cite{holden_robinson}, our starting point is the difficulty of
interpolating the function $f(x) = x^{x^n}$, defined on $x \in \Z$, to
a continuous function on $x \in \Zp$, the ring of $p$-adic integers.  An analytic interpolation
is only possible if the base 
of our $p$-adic exponentiation is in $1 + q \Zp$.  (See for example,
\cite[Section 4.6]{gouvea}, \cite[Section~4.6]{katok}, or
\cite[Section II.2]{koblitz}.)

Therefore, we let $\mu_{\phi(q)}$ be the set of
all $\phi(q)$-th roots of unity contained in $\Zp^\times$, the units in $ \Zp$,  and consider the Teichm\"uller
character
$$\omega: \Zp^\times \to \mu_{\phi(q)},$$
which is a surjective homomorphism.  (Throughout this paper, $\phi(m)$
will refer to the Euler phi function.)  It is known that $\Zp^\times$ has
a canonical decomposition as
\begin{equation} \label{eq:decomposition}
\Zp^\times \cong \mu_{\phi(q)} \times (1+q\Zp)
\end{equation}
\cite[Cor. 4.5.10]{gouvea}, and thus for $x$ in $\Zp^\times$, we may
uniquely write $x = \omega(x) \oneunit{x}$ for some
$\oneunit{x} \in 1 + q\Zp$.

The proof of the following proposition follows from the techniques of Problem~185
of~\cite{gouvea} and Proposition~2.1 of~\cite{holden_robinson}.

\begin{proposition}
  Let $x_0 \in \Z/\phi(q)\Z$, and let
$$I_{x_0} = \set{x \in \Z \mid x \equiv x_0
  \pmod{\phi(q)}} \subseteq \Z.$$  Let $g(x)$ be any polynomial.
Then $$f_{x_0}(x) = \omega(x)^{g(x_0)}\oneunit{x}^{g(x)} =
\omega(x)^{g(x_0)} \exp(g(x) \log \oneunit{x})$$ defines a function
which is analytic on $1 + q\Zp$ and locally analytic
on $\Zp^\times$, such that $f_{x_0}(x) = x^{g(x)}$ whenever $x\in I_{x_0}$.
\end{proposition}

\begin{remark}
  Note that when $p=2$, $I_1 = \Z \setminus 2\Z$, which is dense in
  $\Z_2^\times$.  Therefore we will only need one version of
  $f_{x_0}(x)$, that is, $x_0=1$, in this case.
\end{remark}

Finally, we will want a version of Hensel's lemma
that applies to power series, not just polynomials.  We will use this
in the cases where the solution to an equation is nonsingular modulo $p$.

\begin{definition}[Defn.~III.4.2.2   of~\cite{bourbaki}] 
  A power series $f(x_1,x_2,\dots,x_n)$ in the ring of formal power
  series $\mathbb{Z}_p[[x_1,\dots,x_n]]$ with coefficients in
  $\mathbb{Z}_p$ is called \emph{restricted} if $f(x_1,\dots,x_n)=
  \sum_{(\alpha_i)} C_{\alpha_1,\alpha_2, \cdots, \alpha_n}
  x_1^{\alpha_1}\cdots x_n^{\alpha_n}$ and for every neighborhood $V$
  of 0 in $ \Zp $ there is only a finite number of coefficients
  $C_{\alpha_1,\alpha_2, \cdots, \alpha_n}$ not belonging to $V$ (in
  other words, the family $(C_{\alpha_1,\alpha_2, \cdots, \alpha_n})$
  tends to 0 in $\Zp$).

In particular, the series in this paper are going to be $p$-adic
convergent series $\sum_\alpha C_\alpha x^\alpha$ in $\Zp[[x]]$ such that 
$\lim_{ \alpha \to \infty}  | C_{\alpha} |_p =0$.

\end{definition}


\begin{definition}
Consider a collection of
  $n$ restricted power series $f_j(x_1,x_2,\dots,x_n)$ for $1 \le j\le
  n$ in $\mathbb{Z}_p[[x_1,x_2,\dots,x_n]]$. A vector $(a_1,a_2,\dots,a_n)$ 
   in $\mathbb{Z}_p^n$ is called \emph{nonsingular modulo $p$} if  the determinant of the
  Jacobian matrix at $(a_1,a_2,\dots,a_n)$
$$ \left|  {\frac{\partial(f_1,f_2,\dots,f_n)}{\partial(x_1,x_2, \dots,x_n)}}  
(a_1,a_2,\dots,a_n) \right|$$
is in $\mathbb{Z}_p^\times$.  Otherwise the vector is called \emph{singular modulo $p$}.
\end{definition}

\begin{proposition}[Cor. III.4.5.2
  of~\cite{bourbaki}] \label{hensel-system} Consider a collection of
  $n$ restricted power series $f_j(x_1,x_2,\dots,x_n)$ for $1 \le j\le
  n$ in $\mathbb{Z}_p[[x_1,x_2,\dots,x_n]]$. Let $(a_1,a_2,\dots,a_n)$ be
  a nonsingular vector modulo $p$ such that $f_j(a_1,a_2,\dots,a_n) 
\equiv 0 \pmod p$ for $1\le j \le n$. Then there exists a unique
$(x_1, x_2,\dots,x_n) \in \mathbb{Z}_p^n$ for which $x_i \equiv a_i \pmod p$ 
for $1\le i \le n$ and $f_j(x_1,x_2,\dots,x_n) =0$ in $\mathbb{Z}_p$ for $1 \le j \le n$.
\end{proposition}

As a corollary we get:
\begin{proposition} \label{hensel-one}
Let
  $f(x)$ be a restricted power series in
  $\mathbb{Z}_p[[x]]$, and let ${a}$ be in $\Zp$ such that
  $\frac{df}{dx}(a)$ is in $\Zp^\times$ and $f(a) \equiv 0 \pmod p$.
  Then there exists a unique
$x \in \mathbb{Z}_p$ for which $x \equiv a \pmod p$ 
and $f(x)=0$ in $\mathbb{Z}_p$.
\end{proposition}



\section{Fixed Points} \label{fixed}

In this section, we are concerned with counting roots $x$ of the
function $x^{x^n} - x \bmod{p^e}$, where for a positive integer $e$ and a
prime $p$, we allow $x \in \{ 1, 2, \ldots, p^e(p-1)\} $ such that
$p \nmid x$. To begin, we fix $x_0 \in \Z/(p-1)\Z$ and consider an
auxiliary function $\omega(x)^{g(x_0)} \oneunit{x}^{g(x)} - x \bmod p^e$
defined for any polynomial $g(x)$.

\begin{theorem} \label{x0fixedpt} Let $p$ be a prime $p \neq 2$ and
  $g(x)$ be a polynomial. Then for every $x_0 \in \Z/(p-1)\Z$, there
  are $\gcd(p-1,g(x_0)-1)$ solutions $x$ to the congruence
$$\omega(x)^{g(x_0)} \oneunit{x}^{g(x)} \equiv x \pmod p$$
where $x \in (\Z / p \Z)^\times$.  Alternatively, for any given
$x \in (\Z / p \Z)^\times$, there are
$$N_{g-1}(\ord_p{x}) \frac{p-1}{\ord_p{x}}$$ values of
$x_0 \in \Z/(p-1)\Z$ such that
$$\omega(x)^{g(x_0)} \oneunit{x}^{g(x)} \equiv x \pmod p,$$
where $N_{g-1}(d)$ is the number of solutions to $g(z)-1 \equiv 0$ 
modulo $d$ and $\ord_p{x}$ is the multiplicative order of $x$ modulo $p$.
\end{theorem}

\begin{remark}  For $p=2$ a similar theorem can be proved, but this is not necessary for solving~\eqref{eqn:spfpn}.
\end{remark}

\begin{proof} We know that
  $\oneunit{x} \equiv 1 \pmod{p}$, so the congruence reduces to 
\begin{equation} \label{modp}
\omega(x)^{g(x_0)} \equiv x \pmod{p}.
\end{equation}
For fixed $x_0$, since $\omega(x) \equiv x \pmod p$ by definition, equation (\ref{modp}) has a solution if and only if
$$
\omega(x)^{g(x_0)-1} \equiv 1 \pmod{p}.
$$
This congruence is satisfied for exactly the
$x \in (\Z / p \Z)^\times$ for which $\mbox{ord}_p(x)$ divides $g(x_0)-1$.  There will be
$\gcd(p-1,g(x_0)-1)$ such values for $x$ in the cyclic group
$(\Z / p \Z)^\times$.

On the other hand, if $x$ is fixed, then $\ord_p(x)$ divides
$g(x_0)-1$ if and only if $g(x_0)-1 \equiv 0 \pmod{\ord_p(x)}$.  There
are $N_{g-1}(\ord_p{x})$ such values of $x_0$ in $\Z/(\ord_p x)\Z$ and
$N_{g-1}(\ord_p{x}) ({p-1})/{\ord_p{x}}$ such values of $x_0$ in
$\Z/(p-1)\Z$.
\end{proof}

Next we use the Chinese Remainder Theorem to get the following
corollary to Theorem~\ref{x0fixedpt}. 

\begin{corollary} \label{solnsmodp} Let $p$ be a prime.  Then there are  
$$\sum_{x_0=1}^{p-1} \gcd(p-1,g(x_0)-1) = \sum_{d | p-1} \phi(d) ((p-1)/d) N_{g-1}(d) $$ solutions $x$ to the congruence
$$x^{g(x)}  \equiv x \pmod p$$
where $1 \le x \le p(p-1)$ and $p \nmid x$.
\end{corollary}
\begin{proof}
For $p=2$, this is just the statement that there is one solution modulo $2$.  Otherwise,
Theorem~\ref{x0fixedpt} implies that for each choice of $x_0 \in \Z/(p-1)\Z$, there are $\gcd(p-1, g(x_0)-1)$ elements
$x_1 \in (\Z/p\Z)^\times$ 
with the property that
$$\omega(x_1)^{g(x_0)} \oneunit{x_1}^{g(x_1)} \equiv x_1 \pmod{p}. $$
By the Chinese Remainder Theorem, there will be exactly one  $x \in \Z/p(p-1)\Z$ such that
$x \equiv x_0 \pmod {p-1}$ and $x \equiv x_1 \pmod{p}$. 
By the interpolation we set up in the introduction, since $x \equiv x_0 \pmod{p-1}$, we know that
for each such $x$:
$$x^{g(x)} =\omega(x)^{g(x_0)} \oneunit{x}^{g(x)} \equiv \omega(x_1)^{g(x_0)} \oneunit{x_1}^{g(x_1)} \equiv x_1 \equiv x \pmod {p}. $$
Finally, since exactly $\gcd(p-1, g(x_0)-1)$ such $x$ exist for each $x_0$, we have $\sum_{x_0=1}^{p-1} \gcd(p-1,g(x_0)-1)$ solutions to the congruence.  

Alternatively, for each choice of $x_1 \in (\Z/p\Z)^\times$ of
multiplicative order $d$ modulo $p$, there are $((p-1)/d) N_{g-1}(d)$
values of $x_0 \in \Z/(p-1)\Z$ satisfying the congruence and
$\phi(d)$ choices of $x_1$ with multiplicative order $d$ for each $d
\mid (p-1)$.  (The equality of the two sums also follows from
\cite[Theorem 1]{toth}). 
\end{proof}

 Next we consider $p$-adic solutions to our equation for $x$ such that
 $g(x) \not \equiv 1 \pmod{p}$.  These are the cases where the
 solutions are nonsingular modulo $p$ and thus lift uniquely to solutions modulo $p^e$ and hence to $\Z_p$. We will treat $p \neq 2$ completely and then treat $p=2$.
 
\begin{theorem} \label{nonsinglift} Let $p$ be a prime, $p \neq 2$.  Then there are  
$$
\left\{ \sum_{x_0=1}^{p-1} \gcd(p-1,g(x_0)-1) \right\} - \left\{\sum_{g(x_1)\equiv 1 \pmod{p}} N_{g-1}(\ord_p(x_1)) \frac{p-1}{\ord_p(x_1)}\right\}$$
$$= \sum_{d \mid p-1} \abs{\set{x_1 \in (\Z/p\Z)^\times \mid g(x_1) \not \equiv 1 \pmod{p},\ \ord_p(x_1) = d}}\ \frac{p-1}{d}\ N_{g-1}(d)
$$
solutions $x$ to the congruence 
\begin{equation}\label{modpe}
x^{g(x)}  \equiv x \pmod {p^e}
\end{equation}
where $1 \le x \le p^e(p-1)$ such that $p \nmid x$ and $g(x) \not
\equiv 1 \pmod p$.

These are in one-to-one correspondence with the solutions  $(x,
x_0) \in \Zp \times 
\set{1, \ldots, p-1}$ to the equation
 $$\omega(x)^{g(x_0)}\oneunit{x}^{g(x)} = x$$
such that $p \nmid x$ and $g(x) \not
\equiv 1 \pmod p$.
\end{theorem}
\begin{proof} 

For the cases where $g(x_1) \equiv 1 \pmod p$, $x_1^{g(x_0)-1} \equiv 1 \pmod p$ for all $x_0 \in \Z/(p-1)\Z$
such that $\ord_p (x_1) \mid (g(x_0)-1)$. There will be $N_{g-1}(\ord_p(x_1)) (p-1)/\ord_p(x_1)$ such values of $x_0$.
Now by the Chinese Remainder Theorem, there will be 
the same number of
values for $x$ with $1 \le x \le p(p-1)$
where $p \nmid x$ and $g(x) \equiv 1 \pmod p$.



Fix $x_0 \in \Z/(p-1)\Z$, and consider the function $f_{x_0}: \Z_p \rightarrow \Z_p$ given by $f_{x_0}(x) = \omega(x)^{g(x_0)}\oneunit{x}^{g(x)} - x.$  Note that \begin{eqnarray*}f_{x_0}(x) 
&=& \omega(x)^{g(x_0)}\left(1 + g(x)\log\oneunit{x} + \frac{g(x)^2(\log\oneunit{x})^2}{2!} + \cdots\right) - x.
\end{eqnarray*} 
Now $\log\oneunit{x} \in p\Z_p$, so \begin{eqnarray*}
  f_{x_0}^\prime(x) 
  & \equiv & x^{g(x_0)-1}g(x)
  - 1 \pmod{p} . 
\end{eqnarray*}  

Suppose we have a solution $x_1 \in (\Z/p\Z)^\times$ to
\begin{equation} \label{thm5modp}
 \omega(x)^{g(x_0)} \oneunit{x}^{g(x)} \equiv x \pmod p
 \end{equation}
such that $g(x_1) \not \equiv 1 \pmod{p}$. 
Then
\begin{eqnarray*}f_{x_0}^\prime(x_1) 
& \equiv & g(x_1) - 1 \not\equiv 0\pmod{p}. \end{eqnarray*}
%
By Proposition~\ref{hensel-one}, for fixed $x_0 \in \Z/(p-1)\Z$, each 
$x_1$
will lift to a unique root of
$f_{x_0}(x)$
in $\Zp$. This root in $\Zp$ will correspond to one solution to equation
(\ref{modpe}) for each $e$. Putting these results together with
Corollary~\ref{solnsmodp} and the Chinese Remainder Theorem, and taking out the solutions where
$g(x) \equiv 1 \pmod p$, we have our theorem.

The second summation follows by noting that for each choice of $x_1 \in (\Z/p\Z)^\times$ of multiplicative order $d$ modulo $p$ such that $g(x_1) \not \equiv 1$ modulo $p$, there are $((p-1)/d) N_{g-1}(d)$ values of $x_0 \in \Z/(p-1)\Z$ satisfying the congruence.

\end{proof}

\begin{corollary}\label{nonsingxn}
Let $p$ be a prime $p \neq 2$, then there are 
$$
\left\{ \sum_{x_0=1}^{p-1} \gcd(p-1,x_0^n-1) \right\} - \left\{
  \sum_{x_0=1}^{p-1} \gcd(p-1,n, x_0^n-1) \right\}$$

$$ 
=\left\{\sum_{d \mid p-1} \phi(d)\ \frac{p-1}{d}\ N_{x^n-1}(d)\right\}
- \left\{\sum_{d \mid \gcd(n,p-1)} \phi(d)\ \frac{p-1}{d}\ N_{x^n-1}(d)\right\}
$$
solutions $x$ to the congruence 
\begin{equation}
x^{x^n}  \equiv x \pmod {p^e}
\end{equation}
where $1 \le x \le p^e(p-1)$ such that $p \nmid x$ and $x^n \not
\equiv 1 \pmod p$.

\end{corollary}

\begin{proof}
Let $g(x) = x^n$.  Then for each choice of $x_0$ in Theorem~\ref{x0fixedpt},
there are $\gcd(p-1, n, x_0^n-1)$ elements
$x_1 \in (\Z/p\Z)^\times$ 
with the property that both
$\omega(x)^{x_0^n-1} \equiv 1 \pmod{p}$
and $g(x) = x^n \equiv 1 \pmod p$, since $\omega(x) \equiv x \pmod{p}$, and
thus these are together equivalent to
$$\omega(x)^{\gcd(n, x_0^n-1)} \equiv 1 \pmod{p}.$$

On the other hand, $g(x) \equiv 1$ modulo $p$ is equivalent to $\ord_p(x) \mid n$, which is equivalent to $\ord_p(x) \mid \gcd(n, p-1)$.
So in the previous theorem, $$\abs{\set{x_1 \in (\Z/p\Z)^\times \mid
    g(x_1) \not \equiv 1 \pmod{p},\ \ord_p(x_1) = d}} = \phi(d)$$ if
$d$ divides $p-1$ but not $\gcd(n, p-1)$ and $0$ otherwise. 
\end{proof}

Now we need to specialize exclusively to $g(x) = x^n$ in order to consider the situation when $g(x) \equiv 1 \pmod p$ and the
solutions are singular modulo $p$. Recall that $q=p$ when $p$
 is odd, and $q=4$ when $p=2$. 

\begin{definition} \label{Gae}
Given some $a \in \Z_p$. Let $G_{a,e}$ equal the set of solutions $x$ to the equation 
$$x^{x^n}  \equiv x \pmod {p^e}$$
where $1 \le x \le p^e(p-1)$ such that $p \nmid x$ and $x \equiv a
\pmod q$.
\end{definition}

\begin{definition} \label{Gainf}
Given some $a \in \Z_p$. Let $G_{a, \infty}$ 
  equal the set of solutions $(x, x_0) \in \Zp \times
\set{1, \ldots, p-1}$
 to the equation
$$\omega(x)^{x_0^n}\oneunit{x}^{x^n} = x$$
such that $p \nmid x$ and $x \equiv a \pmod q$.
\end{definition}

\begin{theorem} \label{singlift} Let $p$ be a prime, $p \neq 2$, 
  and let $\xi \in \Zp$ be an $n$th root of unity. Then  
$$
 \left|G_{\xi, e}\right|=\frac{p-1}{\ord_p(\xi)}
 N_{x^n-1}(\ord_p(\xi)) \cdot 
\begin{cases}
p^{e-1} & \text{if~}  e \leq v_p(n)\\  
p^{\lfloor (e+v_p(n))/2 \rfloor} & \text{if~} e \geq v_p(n)+1
\end{cases}$$
and 
$$
 \left|G_{\xi, \infty}\right|=\frac{p-1}{\ord_p(\xi)}
 N_{x^n-1}(\ord_p(\xi)).$$
\end{theorem}

\begin{remark}
Note that in fact the two formulas for $\abs{G_{\xi, e}}$ are equal if $e=v_p(n)+1$ or
$e=v_p(n)+2$.
\end{remark}

\begin{remark}
 If $p \nmid a$ and $a \not \equiv \xi \pmod q$ for $\xi$ equal to some $n$th root of unity in $\Zp$ then see Corollary~\ref{nonsingxn} for $|G_{a,e}|$ and $|G_{a,\infty}|$.
\end{remark}
\begin{proof}
  Consider $x \in \Zp$ such that $x \equiv \xi$ modulo $p$.  Let $1 \leq x_0 \leq p-1,$ and
  let $f_{x_0}(x) = \omega(x)^{x_0^n}\oneunit{x}^{x^n} - x.$ Since we
  are assuming $\xi$ is an $n$th root of unity, we have that
  $\omega(x) = \xi$. We noted in the proof of Theorem~\ref{x0fixedpt} that if
  $\ord_p(\xi)=\ord_p(x)$ does not divide $x_0^n-1$, then there are no
  solutions to $f_{x_0}(x) \equiv 0$ modulo $p$ and thus no solutions
  to $f_{x_0}(x) \equiv 0 \pmod {p^e}$ for any positive integer
  $e$. Thus, we assume that $\ord_p(\xi)$ divides $x_0^n-1$, so
  $\omega(x)^{x_0^n} = \xi^{x_0^n} = \xi$.  There are
  $[(p-1)/\ord_p(\xi)] N_{x^n-1}(\ord_p(\xi))$ such values of
$x_0$.

We have that
\begin{eqnarray} \label{fx0}
f_{x_0}(x) 
%
\notag &=& n\xi^{-1}(x-\xi)^2 + n\  (\text{higher powers of $(x-\xi)$}) 
\end{eqnarray}
and $v_p(f_{x_0}(x)) 
= 2v_p(x-\xi)+v_p(n)$.
Let $\ell = v_p(n)$, the $p$-adic valuation of $n$.  If $1 \leq e \leq \ell+2$, note that for all $x$
such that $1 \leq x \leq p^e$ and $x \equiv \xi \pmod{p}$,
$v_p(f_{x_0}(x)) \geq 2 
+ \ell \geq e$, so $f_{x_0}(x) \equiv 0 \pmod {p^e}$.  There are
$p^{e-1}$ such values of $x$, so there are $p^{e-1}$ solutions to
$f_{x_0}(x) \equiv 0 \pmod {p^e}$ for every solution $\xi$ modulo $p$.

Now we induct on $e$, using $e=\ell+1$ (proved above) as the base case.  
Assume by way of induction that $f_{x_0}(x) \equiv 0$ modulo $p^e$.
Now consider a solution $x$ modulo $p^e$; 
each lifted solution modulo $p^{e+1}$ looks like
$x+tp^e$ for some $0 \leq t < p$.  Modulo $p^{e+1}$, $f_{x_0}(x+tp^e)
\equiv f_{x_0}(x)  + tp^e f'_{x_0}(x)$ by Taylor series expansion
around $x$.  
Then $f_{x_0}(x+tp^e) \equiv 0$ modulo
$p^{e+1}$ if and only if $t f'_{x_0}(x) \equiv -f_{x_0}(x)/p^e$ modulo
$p$.  Since $f'_{x_0}(x) \equiv 0$ modulo $p$, there are
either $p$ solutions, if $f_{x_0}(x)/p^e \equiv 0$ modulo $p$, or no
solutions if not. 

Using the Taylor expansion above, 
$2v_p(x-\xi) = v_p(f_{x_0}(x)) -v_p(n) \geq e-\ell$.  
%
Suppose 
$e-\ell=2k-1$ for some positive integer $k$.  Then
$2v_p(x-\xi) \geq 2k-1$ implies 
$v_p(f_{x_0}(x))\geq 2k+\ell=e+1$.  Thus $p^{e+1} \mid f_{x_0}(x)$, and
$x$ lifts to $p$ solutions modulo $p^{e+1}$. 

Now suppose 
$e-\ell=2k$ for some positive integer $k$.  By
the preceding argument, $v_p(f_{x_0}(x)) -v_p(n)\geq e-\ell = 2k$ if and only
if $v_p(x-\xi) \geq k$, and $v_p(f_{x_0}(x)) -v_p(n)\geq e-\ell+1 = 2k + 1$ if and
only if $v_p(x-\xi) \geq k+1$.  We are assuming
$v_p(f_{x_0}(x)) \geq e$, which thus is equivalent to
$x = \xi + a_k p^k + \alpha p^{k+1}$ for some $0 \leq a_k \leq p-1$
and $\alpha$ in $\Zp$.  Thus $v_p(x-\xi) \geq k+1$ if and only if
$a_k = 0$.  In that case, $x$ lifts to $p$ solutions modulo $p^{e+1}$,
otherwise it does not lift.  In the limit, for each $x_0$ above the only
solution in $\Zp$ 
is where $a_k = 0$ for all $k$, that is, $x = \xi$.

Combining the lifting for $e \geq \ell+1$ with the base case gives
$p^{\floor{(e-\ell)/2}}$ solutions modulo $p^e$ for every solution
modulo $p^{\ell+1}$ and thus
$p^{\floor{(e-\ell)/2}}p^\ell = p^{\floor{(e+\ell)/2}}$ solutions
modulo $p^e$ for each solution modulo $p$, but only one solution in
$\Zp$.

To count $G_{\xi, e}$
we must use the
Chinese Remainder Theorem to argue that for each of the values of
$x_0$ above and for each of the
$p^{\lfloor (e+\ell)/2 \rfloor}$ solutions $x_1$ to $\omega(x)^{x_0^n}x^{x^n}-x \pmod {p^e}$
where $1 \le x_1 \le p^e$ and $x_1 \equiv \xi \pmod p$, there will be
exactly one such $x$ where $1 \le x \le p^e(p-1)$ and
$x \equiv \xi \pmod p$.  The formulas follow.
\end{proof}

Now combining our results from Corollary~\ref{nonsingxn}
and Theorem~\ref{singlift}, we have the following theorem for $p \not = 2$.

\begin{theorem} \label{solnpdivn} Let $p$ be a prime, $p \neq 2$ and $p
  \mid n$.  If $e\leq v_p(n)$, then there are
$$
\left\{ \sum_{x_0=1}^{p-1} \gcd(p-1,x_0^n-1) \right\} + 
\left\{
\sum_{x_0=1}^{p-1} \gcd(p-1,n, x_0^n-1) \cdot \left(p^{e-1}-1\right) \right\}
$$
$$= \left\{ \sum_{d \mid p-1}
  \phi(d)\left(\frac{p-1}{d}\right)N_{x^n-1}(d) \right\} + 
\left\{ \sum_{d \mid \gcd(n,p-1)} \phi(d)\left(\frac{p-1}{d}\right)N_{x^n-1}(d) \cdot
  \left(p^{e-1}-1\right) \right\}
$$
solutions $x$ to the congruence 
\begin{equation*}
x^{x^n}  \equiv x \pmod {p^e}
\end{equation*}
where $1 \le x \le p^e(p-1)$ such that $p \nmid x$.
If $e\geq v_p(n)+1$, then there are
$$
\left\{ \sum_{x_0=1}^{p-1} \gcd(p-1,x_0^n-1) \right\} + 
\left\{
\sum_{x_0=1}^{p-1} \gcd(p-1,n, x_0^n-1) \cdot \left(p^{\lfloor (e+v_p(n))/2
\rfloor}-1\right) \right\}
$$
$$= \left\{ \sum_{d \mid p-1}
  \phi(d)\left(\frac{p-1}{d}\right)N_{x^n-1}(d) \right\} + 
\left\{ \sum_{d \mid \gcd(n,p-1)} \phi(d)\left(\frac{p-1}{d}\right)N_{x^n-1}(d) \cdot
  \left(p^{\lfloor (e+v_p(n))/2 \rfloor}-1\right) \right\}
$$
solutions to the same congruence.  In either case there are  
only  
$$\sum_{x_0=1}^{p-1} \gcd(p-1,x_0^n-1) =  \sum_{d \mid p-1}
  \phi(d)\left(\frac{p-1}{d}\right)N_{x^n-1}(d)$$
solutions  $(x,
x_0) \in \Zp \times 
\set{1, \ldots, p-1}$ to the equation
 $$\omega(x)^{x_0^n}\oneunit{x}^{x^n} = x$$
such that $p \nmid x$.
\end{theorem}

\begin{proof}
This follows directly from
Corollary~\ref{nonsingxn} and Theorem~\ref{singlift}.  (Note that
there are $\gcd(p-1, n)$ elements of $\Zp$ which are $n$th roots of
unity, and each of
them is congruent to a unique integer modulo $p$.)
\end{proof}

When $p=2$, we will see that $f(x)=x^{x^n}-x$ is singular modulo $p$ for all odd values of $x$ where $1 \le x \le p^e$.
The following theorem is analogous to Theorem~\ref{singlift} 
when $p=2$.

\begin{theorem} \label{p-is-2} Let $p=2$, $\xi=\pm 1$ in $\Zp$, and $n$ be a positive integer.
If $n$ is even, we have that
$$|G_{\xi,e}|= 
\begin{cases}
p^{e-2} & \text{if~}  2 \le e \le 4+v_p(n)\\
 p^{\lfloor (e+v_p(n))/2 \rfloor} & \text{if~} e \ge 5+v_p(n)
 \end{cases}
 $$
for all $e \ge 2$.

If $n$ is odd, we have that
$$
|G_{1,e}|= 
\begin{cases}
p^{e-2} & \text{if~}  2 \le e \le 4\\
 p^{\lfloor e/2 \rfloor} & \text{if~} e \ge 5  
 \end{cases},
 \qquad
 |G_{-1,e}|= 
\begin{cases}
p^{e-2} & \text{if~}  2 \le e \le 3\\
 p & \text{if~} e\ge 4  
 \end{cases}
 $$
for all $e \ge 2$.

In all cases, $|G_{\xi, \infty}|=|\set{\xi}|=1$.
\end{theorem}

\begin{remark}
Note that in fact for even $n$ when $\xi=\pm 1$ and for odd $n$ when $\xi=1$, the formulas for $|G_{\xi,e}|$ in the two cases are equal if $e=v_p(n)+3$ or $e=v_p(n)+4$. When $n$ is odd and $\xi=-1$, the two cases are equal if $e=3$. 
\end{remark}

\begin{proof}
    We count solutions for $x \equiv 1 \pmod {q}$ and
  $x\equiv -1 \pmod {q}$ separately.  (Recall that $q=4$ when $p=2$.)
Thus
  $f(x) =\xi \oneunit{x}^{x^n} - x$ 
  when $x \equiv \xi \pmod q$.
  The Taylor series for $f(x)$ centered at
  $\xi$ is
  $$f(x) = 0 + (\xi^n-1)(x-\xi) + \frac{1}{2!}(\xi^{2n-1} +
  (2n-1)\xi^{n-1})(x-\xi)^2 + (\text{higher powers of $(x-\xi)$}) .$$
  If $\xi=1$, the Taylor series reduces to
  $$f(x) = n(x-\xi)^2 + (\text{higher powers of $(x-\xi)$}) $$ for any
  $n$.  If $\xi=-1$ and $n$ is even, the Taylor series is
  $$f(x) = -n(x-\xi)^2 + (\text{higher powers of $(x-\xi)$}). $$
  Finally, if $\xi=-1$ and $n$ is odd, the Taylor series is
  $$f(x) = -2(x-\xi) + (n-1)(x-\xi)^2 + (\text{higher powers of
    $(x-\xi)$}). $$
  Thus we see that $f(x)$ is singular modulo $p$ for all odd $x$ where
  $1 \le x\le p^e$ and all $n$.
   
   After verifying the results for small $e$ and assuming that $f(x) \equiv 0 \pmod {p^e}$, the theorem follows for all odd $x$ by induction on $e$
   in each of the following two cases: (1) $x \equiv 1 \pmod q$ for arbitrary $n$ or $x \equiv -1 \pmod q$ for $n$ even and (2) $x \equiv -1 \pmod q$ for $n$ even. The arguments are similar to those in the proof of Theorem~\ref{singlift}.
\end{proof}
\begin{corollary} \label{p-is-2fixed} If $p=2$ and $n$ is a positive integer, then the number of solutions 
 to the congruence
$$x^{x^n}  \equiv x \pmod{p^e}$$
where $1 \le x \le p^e$ and $p \nmid x$ depends on the valuation $v_p(n)$.

When $n$ is even, the number of solutions is
$$
\begin{cases}
2p^{e-2} & \text{if~}  1 \le e \le 4+v_p(n) \\
2p^{\lfloor (e+v_p(n))/{2} \rfloor} & \text{if~}  e \ge 5+v_p(n).
\end{cases}
$$

When $n$ is odd, the number of solutions is
$$
\begin{cases}
2p^{e-2} & \text{if~} 1 \le e \le 3 \\
p^{\lfloor e/2 \rfloor}+p & \text{if~} e \ge 4.
\end{cases}
$$

In either case the only solutions in $\Zp$ to the equation
 $$\omega(x)\oneunit{x}^{x^n} = x$$
such that $p \nmid x$ are $x=1$ and $x=-1$.
\end{corollary}

\begin{remark}
  Note that in fact for even $n$, the formulas in the two cases modulo
  $p^e$ above are equal if $e=v_p(n)+3$ or $e=v_p(n)+4$, and when $n$
  is odd, the two cases are equal if $e=3$.
\end{remark}

\section{Rooted Closed Walks} \label{sec:rw}

As in~\cite{holden_robinson}, we will address longer cycles from the viewpoint of counting rooted closed walks.  We will later specialize to the case of two-cycles.

\begin{definition} For a fixed prime $p$, the ordered tuple $(x_1, \ldots, x_k)$ is a \emph{rooted
    closed walk of length $k$ modulo $p^e$ associated with the map $x \mapsto x^{g(x)}$}
if the $k$
  equations
\begin{align}\label{eqn:rwmodpe}
 \notag x_1^{g(x_1)} &\equiv x_2 \pmod{p^e},\\
 \notag x_2^{g(x_2)} &\equiv x_3 \pmod{p^e},\\
 &\vdots\\
\notag x_{k-1}^{g(x_{k-1})} &\equiv x_k \pmod{p^e},\\
\notag  x_k^{g(x_k)} &\equiv x_1 \pmod{p^e}
\end{align}
are satisfied.
\end{definition}

For a positive integer $e$ and a prime $p$, we will allow
$x_1, \ldots, x_k \in \{ 1, 2, \ldots, p^e(p-1)\} $ such that
$p \nmid x_i$ for all $i$. We again fix $x_{01},\ldots, x_{0k} \in \Z/(p-1)\Z$ and
consider auxiliary functions
 \begin{equation} \label{eqn:sprwauxfns}
\omega(x_1)^{g(x_{01})} \oneunit{x_1}^{g(x_1)} - x_2 \bmod{p^e}, \quad
\text{\ldots}, \quad \omega(x_k)^{g(x_{0k})} \oneunit{x_k}^{g(x_k)} -x_1 \bmod{p^e}
\end{equation}
defined for a
polynomial $g$.

We will use the isomorphism 
 $$(\Z/p^e\Z)^\times \cong \mu_{p-1} \times (1+p\Z/p^e\Z)$$
 induced from the decomposition~\eqref{eq:decomposition} on
$\Zp^\times$.
This isomorphism tells us that the equations
 \begin{equation} \label{eqn:sprwinterp}
\omega(x_1)^{g(x_{01})} \oneunit{x_1}^{g(x_1)} \equiv x_2 \pmod{p^e}, \quad
\text{\ldots}, \quad \omega(x_k)^{g(x_{0k})} \oneunit{x_k}^{g(x_k)} \equiv x_1 \pmod{p^e}
\end{equation}
are equivalent to the equations
\begin{subequations} \label{eqn:sprwfoureqns}
\begin{equation}\label{eqn:sprwoneunits}
\oneunit{x_1}^{g(x_1)} \equiv \oneunit{x_2} \pmod{p^e}, \quad
\text{\ldots}, \quad  \oneunit{x_k}^{g(x_k)} \equiv \oneunit{x_1} \pmod{p^e},
\end{equation}
\begin{equation}\label{eqn:sprwomegas}
\omega(x_1)^{g(x_{01})} =\omega(x_2), \quad
\text{\ldots}, \quad \omega(x_k)^{g(x_{0k})} = \omega(x_1).
\end{equation}
\end{subequations}


\begin{theorem} \label{x0y0rw} Let $p$ be a prime, $p \neq 2$, and
  $g(x)$ be a polynomial. Then for every $x_{01}, \ldots, x_{0k} \in \Z/(p-1)\Z$, there
  are $\gcd(p-1,g(x_{01})\cdots g(x_{0k})-1)$ solutions $(x_1, \ldots, x_k)$ to the congruences
  \begin{equation} \label{eqn:sprwp}
\omega(x_1)^{g(x_{01})} \oneunit{x_1}^{g(x_1)} \equiv x_2 \pmod{p}, \quad
\text{\ldots}, \quad \omega(x_k)^{g(x_{0k})} \oneunit{x_k}^{g(x_k)} \equiv x_1 \pmod{p}
\end{equation}
where $x_1, \ldots, x_k \in (\Z / p \Z)^\times$.  Alternatively, for any given
$x_k \in (\Z / p \Z)^\times$, there are
$$N_{G-1}(\ord_p{x_k}) \left(\frac{p-1}{\ord_p{x_k}}\right)^k$$ tuples
$(x_{01}, \ldots,x_{0k}) \in (\Z/(p-1)\Z)^k$ such that there exist $x_1, \ldots, x_{k-1} \in  (\Z /
p \Z)^\times$ which solve~\eqref{eqn:sprwp},
where $N_{G-1}(d)$ is the number of solutions to $G(z_1,\ldots, z_k)-1=
g(z_1)\cdots g(z_k)-1 \equiv 0$ modulo $d$. (Note that such $x_1, \ldots, x_{k-1}$ are unique.)
\end{theorem}

\begin{remark}  For $p=2$ a similar theorem can be proved, but this is not necessary for solving~\eqref{rwmodp}.
\end{remark}

\begin{proof} The given congruences are equivalent
  to~\eqref{eqn:sprwfoureqns} with $e=1$, which reduces to just
\begin{equation} \label{rwmodp}
\omega(x_k)^{g(x_{01})\cdots g(x_{0k})-1} =1.
\end{equation}
For fixed $(x_{0i})$, \eqref{rwmodp} is satisfied for exactly the
$x_k \in (\Z / p \Z)^\times$ for which $\mbox{ord}_p(x_k)$, divides
$g(x_{01})\cdots g(x_{0k})-1$.  There will be $\gcd(p-1,g(x_{01})\cdots g(x_{0k})-1)$ such
values for $x_k$ in the cyclic group $(\Z / p \Z)^\times$, and for each
$x_k$ there will be exactly one tuple $(x_1, \ldots, x_{k-1})$ in $((\Z / p \Z)^\times)^{k-1}$ satisfying \eqref{eqn:sprwp}.

On the other hand, if $x_k$ is fixed, then $\ord_p(x_k)$ divides
$g(x_{01})\cdots g(x_{0k})-1$ if and only if $g(x_{01})\cdots g(x_{0k})-1 \equiv 0
\pmod{\ord_p(x_k)}$.  There 
are $N_{G-1}(\ord_p{x_k})$ such tuples $(x_{01}, \ldots ,x_{0k})$ in $(\Z/(\ord_p x_k)\Z)^k$ and
$N_{G-1}(\ord_p{x_k}) (({p-1})/{\ord_p{x_k}})^k$ such tuples in
$(\Z/(p-1)\Z)^k$.  Once again, for each $x_k$, and $x_{01}$, \ldots, $x_{0k}$, the
equations prescribe a unique tuple $(x_1, \ldots, x_{k-1})$.
\end{proof}

\begin{corollary} \label{tcsolnsmodp} Let $p$ be a prime.  Then there are  
$$\sum_{x_{01}=1}^{p-1} \cdots \sum_{x_{0k}=1}^{p-1} \gcd(p-1,g(x_{01})\cdots g(x_{0k})-1) = 
\sum_{d | p-1} \phi(d) ((p-1)/d)^k N_{G-1}(d) $$ solutions $(x_1, \ldots, x_k)$ to the congruences
\begin{equation*}
{x_1}^{g(x_1)} \equiv {x_2} \pmod{p}, \quad
\text{\ldots}, \quad  {x_k}^{g(x_k)} \equiv {x_1} \pmod{p},
\end{equation*}
where $1 \le x_i \le p(p-1)$ and $p \nmid x_i$ for all $i=1,, \ldots, k$.
\end{corollary}
\begin{proof} If $p=2$ then this is just the statement that there is one solution modulo $2$.  Otherwise, the proof follows exactly the proof of Corollary~\ref{solnsmodp}.
\end{proof}



Next we consider solutions modulo $p^e$ and $p$-adic solutions.


\begin{definition} \label{Wabe}
Given $a_1, \ldots, a_k$ in $\Zp$, let $W^k_{a,e}$ equal the set of rooted closed walks of length $k$ modulo $p^e$ associated with the map $x \mapsto x^{g(x)}$
where $1 \le x_i \le p^e(p-1)$, $p \nmid x_i$, and
$x_i \equiv a_i \pmod q$ for all $i=1, \ldots, k$.
\end{definition}


\begin{definition} \label{Wabinf}
Given $a_1, \ldots, a_k$ in $\Zp$. Let $W^k_{a, \infty}$
 equal the set of solutions $(x_1, \ldots, x_k, x_{01}, \ldots, x_{0k}) \in \Z_p^k \times \set{1,
   \ldots, p-1}^k$ to the equations
\begin{equation}
\label{eqn:rwZp}
\omega(x_1)^{g(x_{01})} \oneunit{x_1}^{g(x_1)}= x_2, \quad
\ldots, \quad \omega(x_k)^{g(x_{0k})} \oneunit{x_k}^{g(x_k)} = x_1
\end{equation}
such that $p \nmid x_i$, and $x_i \equiv a_i \pmod q$ for all  $i=1, \ldots, k$.
\end{definition}

We start by identifying and counting the nonsingular solutions.
 Let 
  $h_{1}, \ldots, h_{k}: \Z_p^k \rightarrow \Z_p$ be the functions
\begin{eqnarray*}
 h_{1}(x_1, \ldots, x_k) &= &g(x_1)\log(\oneunit{x_1})-\log(\oneunit{x_2}), \\
 & \vdots &\\
 h_{k}(x_1, \ldots, x_k) &= &g(x_k)\log(\oneunit{x_k})-\log(\oneunit{x_1}).
    \end{eqnarray*}
Note that when \eqref{eqn:sprwomegas} is satisfied, \eqref{eqn:sprwoneunits} is equivalent to
\begin{equation}  \label{eqn:sprwlogeqns}
h_{1}(x_1, \ldots, x_k) \equiv  \cdots \equiv h_k(x_1, \ldots, x_k) \equiv 0 \pmod{p^e},
\end{equation}
 since $z \mapsto \log (z + 1)$
induces a bijection from $p(\Z/p^e\Z)$ to itself, fixing $0$.

We let $J$ denote the Jacobian matrix
$$\begin{pmatrix}
\dfrac{\partial h_1}{\partial x_1} & \cdots & \dfrac{\partial h_1}{\partial x_k} \\ 
\vdots & \ddots & \vdots \\
\dfrac{\partial h_k}{\partial x_1} & \cdots & \dfrac{\partial h_k}{\partial x_k}
\end{pmatrix}$$
$$= 
\begin{psmallmatrix}
\frac{g(x_1)}{x_1}+g'(x_1) \log \oneunit{x_1} & -\frac{1}{x_2} & 0 & \cdots& \cdots & 0\\
0 & \frac{g(x_2)}{x_2}+g'(x_2) \log \oneunit{x_2} & -\frac{1}{x_3} & 0 & \cdots & 0\\
\vdots &\ddots& \ddots & \ddots & \ddots& \vdots \\
0 &\cdots& \cdots & 0 &   \frac{g(x_{k-1})}{x_{k-1}}+g'(x_{k-1}) \log \oneunit{x_{k-1}}  &  -\frac{1}{x_k}\\
-\frac{1}{x_1} & 0 & \cdots & \cdots &0& \frac{g(x_k)}{x_k} + g'(x_k) \log \oneunit{x_k} \\
\end{psmallmatrix}
$$
as usual.  

\begin{theorem} \label{nonsingliftrw} Let $p$ be a prime
  and let $a_1, \ldots, a_k$ be such that $g(a_1) \ldots g(a_k) \not \equiv 1$ modulo
  $p$.  Then $\abs{W^k_{a,e}} =\abs{W^k_{a, \infty}}= \abs{W^k_{a,1}}$
  for all $e\geq 1$.
\end{theorem}

\begin{remark}
Note that if $p=2$ there is only one rooted closed walk modulo $p$.  Whether or not it is singular depends on the value of $g(1)$ modulo $2$.
\end{remark}

\begin{proof}
Suppose we have $z_1, \ldots, z_k \in (\Z/p\Z)^\times$ such that
$g(z_1) \ldots g(z_k) \not \equiv 1 \pmod{p}$.
Note  that 
$\log (1+p\Z_p) \subseteq  p\Z_p$, so
$$\det J(z_1, \ldots, z_k) \equiv \dfrac{g(z_1) \cdots g(z_k)-1}{z_1 \cdots z_k} \not\equiv 0 \pmod{p}.$$

By Proposition~\ref{hensel-system}, for fixed
$(x_{01}, \ldots ,x_{0k}) \in (\Z/(p-1)\Z)^k$, each solution
$(z_1, \ldots, z_k) \in ((\Z/p\Z)^\times)^k$ with
$g(z_1) \cdots g(z_k) \not \equiv 1 \pmod p$ to equations~\eqref{eqn:sprwlogeqns}
will lift to a unique solution 
 in $(\Zp)^k$. Thus this 
will correspond to one solution to equations
(\ref{eqn:sprwinterp}), or equivalently~(\ref{eqn:rwmodpe}), for each $e$. 
Applying the Chinese Remainder Theorem as before gives our result.
\end{proof}

We can count the nonsingular rooted closed walks in a way exactly parallel to Theorem~\ref{nonsinglift}.

\begin{theorem} \label{rwnonsinglift} Let $p$ be a prime.  Then there are  
$$\sum_{x_{01}=1}^{p-1} \cdots \sum_{x_{0k}=1}^{p-1} \gcd(p-1,g(x_{01})\cdots g(x_{0k})-1)
 - \left\{\sum_{g(z_{1})\cdots g(z_{k})\equiv 1 \pmod{p}} N_{G-1}(\ord_p(z_k)) \left( \frac{p-1}{\ord_p(z_k)} \right)^k\right\}$$
$$= \sum_{d \mid p-1} \abs{\set{(z_1, \ldots, z_k) \in ((\Z/p\Z)^\times)^k \mid g(z_1) \cdots g(z_k) \not \equiv 1 \pmod{p},\ \ord_p(z_k) = d}}\left( \frac{p-1}{d}\right)^k N_{G-1}(d)
$$
rooted closed walks of length $k$ modulo $p^e$ associated with the map $x \mapsto x^{g(x)}$
where $1 \le x_i \le p^e(p-1)$ and $p \nmid x_i$ for all $i=1, \ldots, k$, and $g(x_1) \cdots g(x_k) \not
\equiv 1 \pmod p$.

These are in one-to-one correspondence with the solutions  $(x_1, \ldots, x_k,
x_{01}, \ldots, x_{0k}) \in \Zp^k \times 
\set{1, \ldots, p-1}^k$ to~\eqref{eqn:rwZp}
such that $p \nmid x_i$ for all $i=1, \ldots, k$ and $g(x_1) \cdots g(x_k) \not
\equiv 1 \pmod p$.
\end{theorem}

\section{Two-cycles} \label{sec:tc}

We now specialize to the case of $k=2$ and $g(z)=z^n$ in order to count the singular rooted closed walks of length 2, which we will refer to as two-cycles.  We establish a lifting condition using the left kernel of the Jacobian matrix.  This technique appears not to be found in previous literature, although the multivariable Taylor expansion we use is found in Proposition~7.2 of~\cite{lang}.

We will let $x_1=x$, $x_2=y$, $a_1=a$, $a_2=b$, etc.\ in this section, and also use $T_{a,b, \bullet}$ for $W^2_{a,\bullet}$.
In addition to the nonsingular case where $x^n y^n\not\equiv 1$
modulo $p$, there are two singular cases: where $y^n \equiv x^{-n} \not \equiv -1$ modulo $p$, and
where $y^n \equiv x^n \equiv -1$ modulo $p$. 
\begin{theorem} \label{singlifttc} Let $p$ be a prime, $p \neq 2$, 
  and let $a, b \in \Zp$ be roots of unity such that
  $b^n = a^{-n}$.  Then
$$\abs{T_{a,b,e}} =
\begin{cases}
p^{e-1} & \text{if~}  e \leq v_p(n) \text{~and~}  b^n \neq -1\\ 
p^{\lfloor (e+v_p(n))/2  \rfloor} &  \text{if~} e \geq
v_p(n)+1 \text{~and~} b^n \neq -1\\
p^{e-1} & \text{if~}  e \leq 2 v_p(n) \text{~and~}  b^n = -1\\ 
p^{\lfloor (e+v_p(n))/3 \rfloor+\lfloor (e+v_p(n)+1)/3\rfloor}   &
\ \text{if~} e \geq 2v_p(n)+1 \text{~and~}  b^n = -1\\ 
\end{cases}$$
for all $e\geq 1$ and $\abs{T_{a,b,\infty}} =  \abs{T_{a,b,1}}$.
\end{theorem}

\begin{remark}
Note that the powers of $p$ in the first two formulas are the same as
in Theorem~\ref{singlift}, and the the second two formulas are equal
  if $e=2v_p(n)+1$, $e=2v_p(n)+2$, or  $e=2v_p(n)+3$.
\end{remark}

\begin{proof} 
Fix $x_0, y_0 \in \Z/(p-1)\Z$, and consider $x \equiv a$ and $y \equiv b \pmod{p}$.  Assume $(a,b)$ is a solution to~\eqref{eqn:sprwlogeqns} modulo $p$, since otherwise all of the sets are empty.  Also note that if roots of unity $(a,b)$ form a solution modulo $p$, they also form a solution in $\Z_p$. Suppose $(x,y)$ is a solution to~\eqref{eqn:sprwlogeqns} modulo $p^e$.
Each lifted solution then looks like $(x+tp^e, y+up^e)$ for $0 \leq t, u < p$.  Modulo $p^{e+1}$, 
$$\begin{pmatrix} h_1(x+tp^e,y+up^e)\\ h_2(x+tp^e,y+up^e)
\end{pmatrix} \equiv \begin{pmatrix} h_1(x,y)\\ h_2(x,y)
\end{pmatrix} + J \begin{pmatrix} t \\u \end{pmatrix} p^e \pmod{p^{e+1}}
$$
by multivariable Taylor series expansion.  This has a solution $(t,u)$ if and only if $-(h_1(x,y), h_2(x,y))/p^e$ is in the range of $J$ modulo $p$.  Since $(x,y)$ is a singular solution and $J$ is not zero, $J$ must have corank 1.  Then a vector is in the range of $J$ if and only if it is perpendicular to any nonzero vector which spans the left kernel of $J$ modulo $p$, such as $\mathbf{v} = \begin{pmatrix} 1 & a^n \end{pmatrix}$.  Furthermore, if there is a solution $(t,u)$ then there must be exactly $p$ of them.

Using the facts that $x = a\oneunit{x}$ and $y = b\oneunit{y}$ and $b^n=a^{-n}$, our lifting condition is equivalent to the equation
\begin{eqnarray} \label{eqn:sptclifting}
0&\equiv& 
a^n(\oneunit{x}^n-1) \log \oneunit{x} + (\oneunit{y}^n-1)\log \oneunit{y} \pmod{p^{e+1}}
\end{eqnarray}
Since we are assuming $\oneunit{y} \equiv \oneunit{x}^{x^n} \pmod{p^e}$, \eqref{eqn:sptclifting} is then equivalent to
\begin{eqnarray} 
\label{eqn:sptcliftingprime}
0&\equiv& 
a^n\log \oneunit{x} (\oneunit{x}^{nx^n+n}-1) \pmod{p^{e+1}}
\end{eqnarray}
Letting $\bar{h}(x)$ be the right side of this, we have the Taylor expansion
\begin{eqnarray}\label{eqn:sptcliftingtaylor}
\bar{h}(x)&  =  &n a^{n-2}(a^n+1)(x-a)^2 \\
&\quad & \quad + a^{n-3}(n^2 a^n + n^2(a^n+1)^2  -2n(a^n+1))(x-a)^3/2 \notag \\
&\quad & \quad +  n^2\ (\text{higher powers of $(x-a)$}) \notag \\
&\quad &\quad+  n(a^n+1)\ (\text{higher powers of $(x-a)$}) 
\notag
\end{eqnarray}

If $a^n \neq -1$, then $v_p(\bar{h}(x)) = 2v_p(y-b)+v_p(n)$, and we proceed as
in Theorem~\ref{singlift}. Note that the number of solutions of~\eqref{eqn:sprwlogeqns} is the same as the number of
solutions of~\eqref{eqn:sprwoneunits} and that each solution
to~(\ref{eqn:sprwfoureqns}) again gives us a unique solution
to~(\ref{eqn:rwmodpe}) as in Theorem~\ref{nonsingliftrw}.

If $a^n = -1$, then~\eqref{eqn:sptcliftingtaylor} becomes
\begin{eqnarray*}
\bar{h}(x)& = & n^2 a^{-3}(x-a)^3 +  n^2\ (\text{higher powers of $(x-a)$}),
\end{eqnarray*}
and  $v_p(\bar{h}(x)) = 
3v_p(y-b)+2\ell$,
where $\ell = v_p(n)$ as before.
Suppose $1 \leq e \leq 3 + 2\ell$.  If $h_1(a, b) \equiv h_2(a,b) \equiv 0$ modulo
$p$, then for any $x \equiv a \pmod{p}$, $\bar{h}(x) \equiv 0 \pmod{p^e}$ and $(a,b)$ lifts to 
$p^{e-1}$ solutions
modulo $p^e$.  We then induct for
$e \geq 2\ell+1$ as in Theorem~\ref{singlift}, giving us
$p^{\floor{(e-\ell)/3}+\floor{(e-\ell+1)/3}} p^{2\ell} =
p^{\floor{(e+\ell)/3}+\floor{(e+\ell+1)/3}}$
solutions modulo $p^e$ and one solution in $\Zp$ for each solution
modulo $p$.  Applying the Chinese Remainder Theorem then gives us the
result.
\end{proof}

\begin{theorem} \label{tcsolnpnot2} Let $p$ be a prime, $p \neq 2$ and
  $p \nmid n$.  Then there are 
\begin{eqnarray*}
  &&\left.  \sum_{x_0=1}^{p-1}\sum_{y_0=1}^{p-1} \gcd(p-1,x_0^n y_0^n-1)
     \right. \\
  &&\quad + \quad 
     \left.
     \sum_{x_0=1}^{p-1}\sum_{y_0=1}^{p-1} \gcd(p-1, n(y_0^n+1),  x_0^n y_0^n-1) 
     \cdot \left(p^{\lfloor e/2
     \rfloor}-1\right) \right.\\
  &&\quad + \quad 
     \left.
     \sum_{x_0=1}^{p-1}\sum_{y_0=1}^{p-1} \left( \gcd(p-1, 2n,  x_0^n
     y_0^n-1) 
     - \gcd(p-1, n,  x_0^n
     y_0^n-1) \right) \cdot \left(p^{\lfloor e/3 \rfloor+\lfloor (e+1)/3\rfloor} 
     - p^{\lfloor e/2
     \rfloor}\right) \right.\\
  &&= \left. \sum_{d \mid p-1}
     \phi(d)^2\left(\frac{p-1}{d}\right)^2 N_{z^n-1}(d) \right.\\
  && \quad + \quad 
     \left. \sum_{d \mid p-1} \phi(d)\left(\frac{p-1}{d}\right)^2
     \mathcal{N}_{n(z^n+1)}(d)N_{z^n-1}(d) \cdot
     \left(p^{\lfloor e/2 \rfloor}-1\right) \right.\\
  && \quad + \left. \sum_{\substack{d \mid \gcd(p-1,2n) \\ d \nmid \gcd(p-1,n)}}
     \phi(d)^2\left(\frac{p-1}{d}\right)^2 N_{z^n-1}(d)
\right. \cdot
     \left(p^{\lfloor e/3 \rfloor+\lfloor (e+1)/3\rfloor}  
     - p^{\lfloor e/2  \rfloor}\right)
\end{eqnarray*}
solutions $(x,y)$ to the congruences 
\begin{equation*}
x^{x^n}  \equiv y \pmod {p^e}  \quad \text{and} \quad y^{y^n}  \equiv
x \pmod {p^e} 
\end{equation*}
where $1 \le x,y \le p^e(p-1)$ such that $p \nmid x$, $p \nmid y$, 
 ${N}_{z^n-1}(d)$ is the number of solutions to $z^n-1
 \equiv 0$ modulo $d$, and 
 $\mathcal{N}_{n(z^n+1)}(d)$ is the number of solutions to $n(z^n+1)
 \equiv 0$ modulo $d$ such that $z$ is relatively prime to $d$.
 However, there
are only  
$$\sum_{x_0=1}^{p-1}\sum_{y_0=1}^{p-1} \gcd(p-1,x_0^ny_0^n-1) =  \sum_{d \mid p-1}
  \phi(d)^2\left(\frac{p-1}{d}\right)^2N_{z^n-1}(d)$$
  solutions  $(x,y, x_0, y_0) \in \Zp^2 \times 
\set{1, \ldots, p-1}^2$ to the equations
$$\omega(x)^{x_0^n} \oneunit{x}^{x^n}= y \quad
\text{and} \quad \omega(y)^{y_0^n} \oneunit{y}^{y^n} = x$$
such that $p \nmid x$, $p \nmid y$.

\end{theorem}

\begin{remark}
A form of Theorem~\ref{tcsolnpnot2} for $p\mid n$  follows along the lines of
Theorem~\ref{solnpdivn}.  The exact statement is omitted.
\end{remark}

\begin{proof} The total number of solutions modulo $p$ is given by
  Corollary~\ref{tcsolnsmodp}.  A solution $(x,y)$ is in $T_{a,b,1}$
  as in Theorem~\ref{singlifttc} if and only if
  $\omega(y)^{x_0^n y_0^n-1}=1$ and
  $\omega(x)^n\omega(y)^n = \omega(y)^{n(y_0^n+1)}=1$.  (Note that
  $\omega(x)$ and $\omega(y)$ are roots of unity congruent modulo $p$
  to $x$ and $y$, respectively.)  This is equivalent to
  $\omega(y)^{\gcd(n(y_0^n+1),\  x_0^n y_0^n-1)}=1$, and for a fixed
  $x_0$, $y_0$ there are $\gcd(p-1, n(y_0^n+1), x_0^n y_0^n-1)$ such
$y$, each corresponding to a unique $x$ modulo $p$.  Alternatively,
given a $y \in (\Z/p\Z)^\times$ of order $d$, there are 
$$\left(\frac{p-1}{d}\right)^2 \abs{\set{ (x_0, y_0)\in (\set{1, 2,
      \ldots, d})^2 \mid n(y_0^n+1) \equiv 0 \pmod{d}, x_0^n y_0^n-1 \equiv 0
  \pmod{d}}}$$
 pairs $(x_0, y_0) \in (\Z/(p-1)\Z)^2$ satisfying the conditions.
There are $\mathcal{N}_{n(z^n+1)}(d)$ values of $y_0$ in the given set, and
for each one there are $N_{z^n-1}(d)$ values of $x_0$.

Furthermore, a
solution $(x,y)$ is in $T_{a,b,1}$ as above with $b^n = -1$ if and
only if
  $\omega(y)^{x_0^n y_0^n-1}=1$, 
  $\omega(x)^n = \omega(y)^{n(y_0^n)}=-1$, and
  $\omega(y)^n =-1$.  The third condition is equivalent to
  $\omega(y)^{2n}=1$  but $\omega(y)^n \neq 1$, and implies that the
  order of $y$ must be even.  Then the first condition implies that
  $x_0^ny_0^n-1$ must be even, so $x_0$ and $y_0$ must be odd, which
  combined with the third condition makes the second condition
  redundant.  So we have $\omega(y)^{x_0^n y_0^n-1}=1$,
  $\omega(y)^{2n}=1$, and  $\omega(y)^n \neq 1$, which is satisfied
  for $\gcd(p-1, 2n, x_0^n y_0^n-1) - \gcd(p-1, n, x_0^n y_0^n-1)$
  values of $y$ for each fixed pair $(x_0, y_0)$.  Alternatively, the
  conditions imply that for 
for
  each $y \in (\Z/p\Z)^\times$ of order $d$, $d$ must divide $2n$ but
  not $n$, and if so there are 
$$\left(\frac{p-1}{d}\right)^2 \abs{\set{ (x_0, y_0)\in (\set{1, 2,
      \ldots, d})^2 \mid x_0^n y_0^n-1 \equiv 0
  \pmod{d}}}$$
 pairs $(x_0, y_0) \in (\Z/(p-1)\Z)^2$ satisfying the conditions.
There are $\phi(d)$ values of $y_0$ in the given set, and
for each one there are $N_{z^n-1}(d)$ values of $x_0$.
\end{proof}

When $p=2$, we see that, as in the fixed point case, our equation is
singular modulo $p$ for all odd values of $x$.  However, this time the
lifting only takes two different forms, rather than three forms as in
Theorem~\ref{p-is-2}.

\begin{theorem} \label{p-is-2tc} Let $p=2$. Then in all cases when $a  \not =b$ there are no two-cycles. However,
when $n$ is even and $a=\pm 1$ or when $n$ is odd and $a=1$, we have that 
$$\abs{T_{a,a,e}} = 
\begin{cases}
p^{e-2} & \text{if~}  2 \leq e \leq v_p(n)+4\\ 
p^{\lfloor (e+v_p(n)+1)/2  \rfloor} &  \text{if~} e \geq
v_p(n)+5 
\end{cases}$$
for all $e\geq 2$. 
\noindent
And when $n$ is odd and $a=-1$, we have that
$$\abs{T_{a,a,e}} = 
\begin{cases}
p^{e-2} & \text{if~}  2 \leq e \leq 4\\ 
p^{\lfloor e/3 \rfloor + \lfloor (e+1)/3  \rfloor} &  \text{if~} e \geq 5 
\end{cases}$$
for all $e\geq 2$. 
In all cases, $|T_{a, a, \infty}| = |\set{(a,a)}|=1$.
\end{theorem}

\begin{remark}
  Note that the powers of $p$ in each of the two cases modulo $p^e$
  above are equal if $e=v_p(n)+4$ or $e=v_p(n)+5$.
\end{remark}

\begin{proof}
  The proof is essentially the same as Theorem~\ref{singlifttc} except for an extra factor of 2 in the Taylor expansion when $a = \pm 1$ and $n$ is even or when $a=1$ and $n$ is odd.
\end{proof}

\begin{corollary} Let $p=2$ .
Then the number of solutions $(x,y)$ to the congruences
  \begin{equation*}
x^{x^n}  \equiv y \pmod {p^e}  \quad \text{and} \quad y^{y^n}  \equiv
x \pmod {p^e} 
\end{equation*}
 for $1 \leq y \leq p^e$ where $p \nmid x$, $p \nmid y$ and $p \nmid n$ is 
 $$ 
\begin{cases}
1 &  \text{if~}   e =1  \\ 
2p^{e-2} & \text{if~}  2 \le e \le 4  \\ 
p^{\lfloor (e+1)/2  \rfloor} + p^{\lfloor e/3 \rfloor + \lfloor (e+1)/3  \rfloor} &  \text{if~} e \geq 5 
\end{cases}$$
for all $e\geq 1$.  
And when $p \mid n$
$$
\begin{cases}
1 &  \text{if~}   e =1  \\ 
2p^{e-2} & \text{if~}  2 \le e \leq v_p(n)+4  \\ 

2p^{\lfloor (e+v_p(n)+1)/2  \rfloor} &  \text{if~} e \geq
v_p(n)+5 
\end{cases}
$$
for all $e\ge 1$. For all $n$ the only solutions in $\Zp^2$ to the equations
$$\omega(x) \oneunit{x}^{x^n}= y \quad
\text{and} \quad \omega(y) \oneunit{y}^{y^n} = x$$
such that $p \nmid x$, $p \nmid y$ are $(x,y)=(1,1)$ and $(x,y)=(-1,-1)$. 
\end{corollary}
\begin{remark}
Note that the formulas for solutions modulo $p^e$ above are equal
  if $e=v_p(n)+4$ or if $e=v_p(n)+5$.
\end{remark}

\section{Future Work} \label{future}

Extending the results of Section~\ref{sec:tc} to rooted closed walks of size three and larger does not
seem in principle like it would present any difficulties.  The matrix $J$ will still have corank 1 and equation~\eqref{eqn:sptclifting} extends in a fairly straightforward way. Equation~\eqref{eqn:sptcliftingprime} then becomes
$$0 \equiv a_1^n \log \oneunit{x_1} (\oneunit{x_1}^{n(1+x_1^n+x_1^nx_2^n + \cdots + x_1^n \cdots x_{k-1}^n)}-1) \pmod{p^e}
$$
which seems potentially difficult to deal with in practice.

Another significant advance in the case of points
that are singular modulo $p$ would be to extend
more of these results to the generalized self-power map $x \mapsto
x^{g(x)}$ for any polynomial $g(x)$.  Our results can be used to
count solutions modulo $p$ for any polynomial.  We can also determine
which solutions are nonsingular modulo $p$ and thus lift
uniquely.
On the other hand, we are not able to count the lifts that are singular modulo $p$ without
using a fairly specific form of the polynomial.  Extending to $g(x) =
cx^n$ seems like a reasonable next case to try.

Two other types of congruences modulo $p^e$ involving the self-power
map were studied in~\cite{holden_robinson}, namely 
$x^x \equiv c \pmod{p^e}$ and  $x^x \equiv y^y \pmod{p^e}$.
These could also be
generalized to the expression $x^{g(x)}$ studied here.  In the case of
$x^x$, these expressions are always nonsingular modulo $p$, but for some
polynomials $g(x)$, this will no longer be the case.

This work explores solutions to our equations in the range $\set{1,
  \dots, p^e(p-1)}$.  For cryptographic
applications, we would be most interested in solutions in the range $\set{1,
  \dots, p^e}$.  If $p=2$, these are the same, and we find that we can
both count and (by following the proofs) describe completely our
solutions. 
 For applications where we wish to take advantage of
pseudorandom properties of functions, this suggests that variations on
the self-power map may not be appropriate when $p=2$.  It is possible that this
predictability might be an advantage for other
applications.

For $p > 2$, the standard heuristics suggest that the behavior modulo
$p-1$ of $x \in \set{1, \dots, p^e(p-1)}$ is ``independent'' of the
behavior modulo $p^e$.  Thus, for example, if a fixed point $x \in \set{1,
  \dots, p^e(p-1)}$ comes via 
the Chinese Remainder theorem from a pair $(x_0, x_1) \in \Z/(p-1)\Z
\times \Z/p^e\Z$, we would expect approximately $1/(p-1)$ of such fixed points to
work out so that $x \in \set{1,  \dots, p^e}$.  (See, for example,
\cite{holden_moree} for similar heuristics.)  This suggests that the
numbers of fixed points and two-cycles in $\set{1,  \dots, p^e}$
should have some distribution centered around $1/(p-1)$ times the numbers
calculated in this paper.  Some experimental results on this and related
distributions in the case $e=1$ may be found in
\citelist{\cite{REU2010}*{Section~8}\cite{holden_friedrichsen}\cite{kurlberg_et_al}*{Section~4}\cite{felix_kurlberg}*{Section~1.2}}. We
are not aware of any similar results for $e > 1$.


















\section*{Acknowledgements}
The second author would like to thank the Hutchcroft Fund at
  Mount Holyoke College for support and the Department of Mathematics
  at Mount Holyoke for their hospitality during a visit in the spring
  of 2015.  We would also like to thank Eric Bach for pointing out the similarity of our Taylor expansion to Proposition~7.2 of~\cite{lang}.


\begin{bibdiv}
\begin{biblist}

\bib{anghel}{thesis}{ 
  title={The Self Power Map and its Image Modulo
      a Prime},
    url={https://tspace.library.utoronto.ca/handle/1807/35765},
    type={PhD Thesis}, 
    school={University of Toronto}, 
    author={Anghel, Catalina Voichita}, 
    year={2013},
}

\bib{anghel16}{article}{
  title={The self-power map and collecting all    residue classes}, 
  volume={85}, 
  ISSN={0025-5718, 1088-6842},
  DOI={10.1090/mcom/2978}, 
  number={297},
  journal={Mathematics of Computation}, 
  author={Anghel, Catalina},
  year={2016}, 
  pages={379--399}
}




\bib{balog_et_al}{article}{
	title = {On the number of solutions of exponential congruences},
	author = {Antal Balog},
        author = {Kevin A. Broughan},
        author = {Igor E. Shparlinski},
	volume = {148},
	url = {http://journals.impan.gov.pl/aa/Inf/148-1-7.html},
	doi = {10.4064/aa148-1-7},
	number = {1},
	journal = {Acta Arithmetica},
	year = {2011},
	pages = {93--103}
}

\bib{bourbaki}{book}{
	edition = {1},
	title = {Commutative Algebra: Chapters 1-7},
	shorttitle = {Commutative Algebra},
	publisher = {Addison-Wesley},
	author = {Nicolas Bourbaki},
	year = {1972}
}




\bib{CS97}{inproceedings}{
  place={Berlin, Heidelberg},
  title={Efficient group signature schemes for large groups},
  volume={1294}, 
  ISBN={978-3-540-63384-6},
  url={http://www.springerlink.com/content/l31j7157g5102122/},
  author={Camenisch, Jan},
  author = {Stadler, Markus},
    booktitle={Advances in Cryptology --- CRYPTO '97}, 
    book={
        publisher={Springer},
        place={Berlin Heidelberg}, 
        editor={Kaliski, Burton S.}, 
        year={1997}, 
    },
    pages={410--424}
}



\bib{cheon}{article}{
  title={Discrete logarithm problems with auxiliary inputs}, 
  volume={23}, 
  ISSN={0933-2790, 1432-1378},
  DOI={10.1007/s00145-009-9047-0}, 
  number={3},
  journal={Journal of Cryptology}, 
  author={Cheon, Jung Hee},
  year={2009}, 
  pages={457--476}
}


\bib{CG14}{article}{
  title={Congruences involving product of intervals and sets with small
    multiplicative doubling modulo a     prime and applications}, 
    volume={160}, 
    ISSN={0305-0041, 1469-8064},  
    number={3}, 
    journal={Mathematical Proceedings of the Cambridge Philosophical Society}, author={Cilleruelo, J.},
    author ={Garaev, M. Z.}, 
    year={2016}, 
    pages={477--494}
}


\bib{CG15}{article}{
  title={The congruence  $x^x\equiv \lambda \pmod p$},
   author={Cilleruelo, J.},
  author = {Garaev, M.},
    volume={144}, 
    number={6}, 
    journal={Proceedings of the American Mathematical Society}, 
    year={2016}, 
    pages={2411--2418}
}






\bib{crocker66}{article}{
	title = {On a new problem in number theory},
	volume = {73},
	number = {4},
	journal = {The American Mathematical Monthly},
	author = {Crocker, Roger},
	year = {1966},
	pages = {355--357},
}

\bib{crocker69}{article}{
	title = {On Residues of $n^n$},
	volume = {76},
	number = {9},
	journal = {The American Mathematical Monthly},
	author = {Crocker, Roger},
	year = {1969},
	pages = {1028--1029},
}



\bib{felix_kurlberg}{article}{
  title={On the fixed points of the map $x \mapsto x^x$ modulo a prime, II},
  author={Felix, Adam Tyler},
  author = {Kurlberg, P\"ar},
 volume={48}, 
  journal={Finite Fields and Their Applications}, year={2017}, 
    pages={141--159}
}





\bib{G18}{article}{
title={On distribution of elements of subgroups in arithmetic progressions modulo a prime}, volume={303}, 
number={1}, 
journal={Proceedings of the Steklov Institute of Mathematics}, 
author={Garaev, M. Z.}, 
year={2018}, 
pages={50--57} 
}

\bib{HHJ}{article}{
title={On the congruence $x^x \equiv x$ (mod $n$)}, 
volume={16}, 
journal={Integers}, 
author={Hammer, James},
author ={Harrington, Joshua},
author = {Jones, Lenny}, 
year={2016}, 
number = {A74},
pages={1--17} 
}


\bib{holden_friedrichsen}{article}{ 
  author={Friedrichsen, Matthew},
  author = {Holden, Joshua}, 
  title={Statistics for fixed points of the self-power map}, 
    volume={12}, 
    number={1}, 
    journal={Involve, a Journal of Mathematics},  year={2019}, 
    pages={63--78}
    }


\bib{REU2010}{article}{
	title = {Structure and statistics of the {self-power} map},
	volume = {11},
	number = {2},
	journal = {{Rose-Hulman} Undergraduate Mathematics Journal},
	author = {Friedrichsen, Matthew},
        author = {Larson, Brian},
        author = {McDowell, Emily},
	year = {2010},
}





\bib{gouvea}{book}{
	edition = {2},
	title = {p-adic Numbers: An Introduction},
	isbn = {3540629114},
	shorttitle = {p-adic Numbers},
	publisher = {Springer},
	author = {Fernando Quadros Gouvea},
	month = {jul},
	year = {1997}
}


\bib{holden02}{inproceedings}{
      author={Holden, Joshua},
       title={Fixed points and two-cycles of the discrete logarithm},
	date={2002},
   booktitle={Algorithmic number theory ({A}{N}{T}{S} 2002)},
      editor={Fieker, Claus},
      editor={Kohel, David~R.},
      series={LNCS},
   publisher={Springer},
       pages={405\ndash 415},
  url={http://link.springer-ny.com/link/service/series/0558/bibs/2369/23690405%
.htm},
}

\bib{holden02a}{article}{
      author={Holden, Joshua},
       title={Addenda/corrigenda: Fixed points and two-cycles of the discrete
  logarithm},
	date={2002},
status = {Unpublished}, 
       eprint = {arXiv:math/020802 [math.NT]},
}



\bib{holden_moree}{article}{
	title = {Some heuristics and results for small cycles of the
          discrete logarithm}, 
	volume = {75},
	issn = {0025-5718},
	number = {253},
	journal = {Mathematics of Computation},
	author = {Joshua Holden},
        author = {Pieter Moree},
	year = {2006},
	pages = {419--449}
}


\bib{holden_robinson}{article}{
	title = {Counting fixed points, {two-cycles}, and collisions
          of the discrete exponential function using $p$-adic
          methods}, 
	journal = {Journal of the Australian Mathematical Society},
	author = {Holden, Joshua},
        author = {Robinson, Margaret M.},
        volume = {92},
        number = {2},
        pages = {163--178},
        year = {2012}
}


\bib{katok}{book}{
  place={Providence,  RI},
  title={$p$-adic Analysis Compared with Real}, 
  ISBN={978-0-8218-4220-1}, 
  publisher={American Mathematical Society},
  author={Katok, Svetlana}, 
  year={2007}
} 


\bib{koblitz}{book}{
	edition = {2},
	title = {$p$-adic Numbers, $p$-adic Analysis, and
          {Zeta-Functions}},
        series = {Graduate Texts in Mathematics},
	isbn = {0387960171},
	publisher = {Springer},
	author = {Neal Koblitz},
	month = {jul},
	year = {1984}
}

\bib{kurlberg_et_al}{article}{
    author = {Kurlberg, P\"ar},
    author = {Luca, Florian},
    author = {Shparlinski, Igor E.},
    title = {On the fixed points of the map $x \mapsto x^x$ modulo a
      prime},
    journal = {Mathematical Research Letters},
    volume =  {22},
    number = {1},
    year = {2015},
    pages = {141--168},
    doi = {10.4310/MRL.2015.v22.n1.a8}
}


\bib{lang}{book}{
place={New York, NY}, 
series={Graduate Texts in Mathematics}, title={Algebra}, publisher={Springer}, author={Lang, Serge}, 
year={2002}, 
pages={465--499}, 
collection={Graduate Texts in Mathematics},
volume={211},
edition = {3}
}






\bib{mann}{thesis}{
  title={Counting Solutions to Discrete
    Non-Algebraic Equations Modulo Prime Powers},
  number={Mathematical Sciences Technical Reports (MSTR) 153},
  organization={Rose-Hulman Institute of Technology}, 
  author={Mann,  Abigail}, 
  date={2016}, 
  type = {Senior Thesis}
}



\bib{handbook}{book}{
	title = {Handbook of Applied Cryptography},
	isbn = {0849385237},
	url = {http://www.cacr.math.uwaterloo.ca/hac/},
	publisher = {{CRC}},
	author = {Alfred J. Menezes},
        author = {Paul C. van Oorschot},
        author = {Scott A. Vanstone},
	month = {oct},
	year = {1996}
}

\bib{somer}{article}{
  author = {Somer, Lawrence},
  title = {The residues of $n^n$ modulo $p$},
  journal = {Fibonacci Quarterly},
  volume = {19},
  number = {2},
  year = {1981},
  pages = {110--117}
}  

\bib{stadler}{inproceedings}{
  series={Lecture Notes in Computer Science}, 
  title={Publicly verifiable secret sharing}, 
    ISBN={978-3-540-61186-8},
 author={Stadler, Markus},
    booktitle={Advances in Cryptology --- {EUROCRYPT} '96},
   url={http://link.springer.com/chapter/10.1007/3-540-68339-9_17},
    book ={  
    publisher={Springer},
    place={Berlin Heidelberg}, 
    editor={Maurer, Ueli}, 
    year={1996},
    },   
        pages={190--199}
}


\bib{toth}{article}{
	title = {Menon's identity and arithmetical sums representing functions of several variables},
	volume = {69},
	issn = {0373-1243},
	url = {http://www.ams.org/mathscinet-getitem?mr=2884710},
	number = {1},
	urldate = {2015-08-22},
	journal = {Rendiconti del Seminario Matematico. Università e
          Politecnico Torino}, 
	author = {T\'oth, L.},
	year = {2011},
	mrnumber = {2884710},
	pages = {97--110}
}

\bib{wood}{report}{
  author = {Wood, A.},
  title = {The square discrete exponentiation map},
  institution = {Rose-Hulman Institute of Technology},
  series =  {Mathematical Sciences Technical Reports},
  year = {2011},
  number = {MSTR 11-05},
  eprint = {http://scholar.rose-hulman.edu/math_mstr/9/},
  url = {http://scholar.rose-hulman.edu/math_mstr/9/}
}



\end{biblist}
\end{bibdiv}
\end{document}